\documentclass[10pt,reqno]{amsart}
\usepackage{amsmath,amscd,amsfonts,amsthm,amssymb,latexsym,mathrsfs}
\usepackage{graphicx,epsf,epic,eepic,epsfig}
\usepackage{url}

\theoremstyle{plain}
   \newtheorem{theorem}{Theorem}[section]
   
   \newtheorem{lemma}[theorem]{Lemma}
   \newtheorem{corollary}[theorem]{Corollary}
   \newtheorem{problem}[theorem]{Problem}
   \newtheorem{conjecture}[theorem]{Conjecture}
   
\theoremstyle{definition}

\theoremstyle{remark}
   \newtheorem{remark}[theorem]{Remark}

\author[P.~Br\"and\'en]{Petter Br\"and\'en}
\address{Department of Mathematics, Royal Institute of Technology, SE-100 44 Stockholm,
Sweden}
\email{pbranden@math.kth.se}


\thanks{PB is a Royal Swedish Academy of Sciences Research Fellow
  supported by a grant from the Knut and Alice Wallenberg
  Foundation.}
\def\kkk{\kern.2ex\mbox{\raise.5ex\hbox{{\rule{.35em}{.12ex}}}}\kern.2ex}

\newcommand{\nn}{\mathbf{n}}

\newcommand{\NN}{\mathbb{N}}

\newcommand{\RR}{\mathbb{R}}
\newcommand{\CC}{\mathbb{C}}

\newcommand{\sym}{\mathfrak{S}}

\def\newop#1{\expandafter\def\csname #1\endcsname{\mathop{\rm
#1}\nolimits}}

\newop{per}
\newop{diag}
\newop{supp}
\newop{Proj}
\newop{JP}
\newop{rank}
\newop{Sym}
\newop{sign}
\newop{Int}
\newop{Cap}
\newop{PolO}
\newop{St}
\newop{det}
\newop{MAP}
\newop{MOD}

\begin{document}

\title[Two problems on permanents]{Solutions to two problems on permanents}

\maketitle

\begin{abstract}
In this note we settle two open problems  in the theory of permanents by using recent results from other areas of mathematics. 
Both problems were recently discussed in Bapat's survey \cite{Bapat2}. Bapat conjectured that certain quotients of permanents, which generalize symmetric function means, are concave. We prove this conjecture by using concavity properties of hyperbolic polynomials. Motivated by problems on random point processes, Shirai and Takahashi raised the problem: Determine all real numbers  $\alpha$ for which the $\alpha$-permanent (or $\alpha$-determinant) is nonnegative for all positive semidefinite matrices. We give a complete solution to this problem by using recent results of Scott and Sokal on completely monotone functions.  It turns out that the conjectured answer to the problem is false. 
\end{abstract}

\section{Bapat's conjecture on quotients of permanents}

Recently Gurvits \cite{Gurvits} successfully used hyperbolic polynomials to prove inequalities for permanents and determinants. In this section we   
show how a conjecture (Conjecture \ref{bap} below) of Bapat on the concavity of certain quotients of permanents follows from concavity properties of hyperbolic polynomials. Recall that  if $A=(a_{ij})_{i,j =1}^n$ is a matrix, then  the \emph{permanent} of $A$ is defined by 
$$
\per(A) = \sum_{\sigma \in \sym_n} \prod_{i=1}^n a_{i\sigma(i)}, 
$$
where $\sym_n$ is the symmetric group on $\{1,\ldots, n\}$. 
\begin{conjecture}[Bapat, \cite{Bapat1}]\label{bap}
Let $b_0, b_1, \ldots, b_k$ be fixed vectors in $\RR_{++}^n:=(0,\infty)^n$, where $0\leq k<n$. The function 
\begin{equation}\label{bapeq}
x \mapsto \frac {\per(b_1, \ldots, b_k, x,\ldots, x)}{\per(b_0,b_1, \ldots, b_k, x,\ldots, x)}
\end{equation}
is concave on $\RR_{++}^n$. 
\end{conjecture}
A motivation for Conjecture \ref{bap} is the case when $b_0, b_1, \ldots, b_k$  are all equal to the vector of all ones. Then \eqref{bapeq} is equal to 
a constant multiple of 
\begin{equation}\label{kvo}
x \mapsto \frac{e_{n-k}(x)}{e_{n-k-1}(x)}, 
\end{equation}
where $e_k(x)$ is the $k$'th \emph{elementary symmetric function} in the variables $x=(x_1,\ldots, x_n)$. The function \eqref{kvo} is a \emph{symmetric function mean} and such are  known to be concave \cite{ML}. 

A homogeneous polynomial $h(x) \in \RR[x_1, \ldots, x_n]$ is \emph{hyperbolic} with respect to a vector 
$e \in \RR^n$ if $h(e) \neq 0$, and if for all $x \in \RR^n$ the univariate polynomial $t \mapsto h(x+et)$ has only real 
zeros, see \cite{BGLS, Garding, Ren1, Renegar}. Here are some examples of hyperbolic polynomials:
\begin{enumerate}
\item Let $h(x)= x_1\cdots x_n$. Then $h(x)$ is hyperbolic with respect to any vector $e \in \RR^n$ that has no coordinate equal to zero: 
$$
h(x+et) = \prod_{j=1}^n (x_j+e_jt).
$$
\item Let $x=(x_{ij})_{i,j=1}^n$ be a matrix of variables where we impose $x_{ij}=x_{ji}$. Then $\det(x)$ is hyperbolic with respect to 
$I=\diag(1, \ldots, 1)$. Indeed $t \mapsto \det(x+tI)$ is the characteristic polynomial of the symmetric matrix $x$, so it has only real zeros. 
\item Let $h(x)=x_1^2-x_2^2-\cdots-x_n^2$. Then $h$ is hyperbolic with respect to $(1,0,\ldots,0)$. 
\end{enumerate}

Suppose that $h$ is hyperbolic with respect to $e$, and of degree $d$. We may write 
$$
h(x+et) = h(e)\prod_{j=1}^d (t + \lambda_j(x)),
$$
where $\lambda_1(x) \leq \cdots \leq \lambda_d(x)$. The \emph{hyperbolicity cone} is the set 
$$
\Lambda_{\tiny{++}}= \Lambda_{\tiny{++}}(e)= \{ x \in \RR^n : \lambda_1(x) >0\}.
$$
G\aa rding \cite{Garding} proved that hyperbolicity cones are convex. The hyperbolicity cones for the examples  above are:
\begin{enumerate}
\item $\Lambda_{\tiny{++}}(e)= \{ x \in \RR^n : x_ie_i>0 \mbox{ for all } i\}$. 
\item $\Lambda_{\tiny{++}}(I)$ is the cone of symmetric positive definite matrices.
\item $\Lambda_{\tiny{++}}(1,0,\ldots,0)$  is the \emph{Lorentz cone} 
$$
\left\{x \in \RR^n : x_1 > \sqrt{x_2^2+\cdots+x_n^2}\right\}.
$$
\end{enumerate}
 Let $h(x_1,\ldots, x_n)$ be a homogeneous polynomial of degree $d$. Let $v_j= (v_{1j},\ldots, v_{nj})^T$ for $1\leq j \leq d$. The \emph{complete polarized form} of $h$ may be defined as 
 the form $H : (\RR^n)^d \rightarrow \RR$ defined by 
 $$
 H(v_1,\ldots, v_d) =  \frac 1 {d!} \prod_{j=1}^d \left(\sum_{i=1}^n v_{ij}\frac \partial {\partial x_i}\right) h(x),  
 $$
see G\aa rding \cite{Garding}. Note that $H$ is multilinear,  symmetric in $v_1,\ldots, v_d$, and $H(v,\ldots, v)= h(v)$. Now if $h=x_1\cdots x_n$, then 
\begin{equation}\label{homper}
H(v_1,\ldots, v_n)= \frac 1 {n!} \per(v_1,\ldots, v_n).
\end{equation}
For $u =(u_1,\ldots, u_n)^T \in \RR^n$ let
$
D_u = \sum_{i=1}^n u_i {\partial} /{\partial x_i}.
$
It follows that 
\begin{align*}
H(v_1,\ldots, v_k, u, \ldots, u) &= \frac {(d-k)!}{d!} \frac 1 {(d-k)!} \prod_{j=1}^{d-k} \left(\sum_{i=1}^n u_{i}\frac \partial {\partial x_i}\right)D_{v_1}D_{v_2}\cdots D_{v_k}h(x)\\
&= \frac {(d-k)!}{d!} D_{v_1}D_{v_2}\cdots D_{v_k}h(u).
\end{align*}

\begin{lemma}[G\aa rding, \cite{Garding}]\label{internat}
Suppose that $h$ is hyperbolic, and that $v \in \Lambda_{++}$. Then $D_vh$ is hyperbolic and its hyperbolicity cone contains $\Lambda_{++}$.  
\end{lemma}
The following lemma was proved in \cite[Corollary 4.6]{BGLS}, see also \cite{Renegar} where Lemma~\ref{conca} is strengthened. 
\begin{lemma}[Bauschke \emph{et al.}, \cite{BGLS}]\label{conca}
Suppose that  $h$ is hyperbolic and that $v \in \Lambda_{++}$. The function 
$$
x \mapsto \frac {h(x)}{D_vh(x)} 
$$
is concave on $\Lambda_{++}$. 
\end{lemma}

In view of \eqref{homper} we see that Conjecture \ref{bap} is the special case of Corollary \ref{conform} when $h=x_1\cdots x_n$. 

\begin{corollary}\label{conform}
Let $h$ be a hyperbolic polynomial in $\RR[x_1,\ldots, x_n]$ and let $b_0, b_1, \ldots, b_k$ be fixed vectors in $\Lambda_{++}$. The function 
\begin{equation}\label{bapeq2}
x \mapsto \frac {H(b_1, \ldots, b_k, x,\ldots, x)}{H(b_0,b_1, \ldots, b_k, x,\ldots, x)}
\end{equation}
is concave on $\Lambda_{++}$. 
\end{corollary}

\begin{proof}
Suppose that the degree of $h$ is $d$. By Lemma \ref{internat} the polynomial
$$
g(x):= H(b_1, \ldots, b_k, x,\ldots, x)=\frac {(d-k)!}{d!} D_{b_1}\cdots D_{b_k}h(x)
$$
is hyperbolic with hyperbolicity cone containing $\Lambda_{++}$. The function \eqref{bapeq2} is equal to 
$g(x)/D_{b_0}g(x)$, and so the proof follows from Lemma \ref{conca}. 
\end{proof}
\begin{remark}
If $h(x)= \det(x)$, acting on symmetric matrices of size $n \times n$, then the complete homogenized form is the \emph{mixed discriminant}
$$
H(A_1, \ldots, A_n)= \frac 1 {n!} \frac {\partial^n}{\partial x_1\cdots \partial x_n} \det\left(\sum_{i=1}^n x_iA_i\right).
$$
Hence if $A_0, \ldots, A_k$ are fixed positive definite matrices, then the function 
$$
A \mapsto  \frac {H(A_1, \ldots, A_k, A,\ldots, A)}{H(A_0,A_1, \ldots, A_k, A,\ldots, A)}
$$
is concave on the cone of positive definite matrices. This also holds for complex hermitian matrices since the determinant on complex hermitian matrices is again hyperbolic. 
\end{remark}

\section{ $\alpha$-permanents and complete monotonicity}
The $\alpha$-permanent, introduced by Vere--Jones \cite{VJ}, interpolates between the determinant and the permanent.  
Let $\alpha \in \RR$ and $A=(a_{ij})$ be an $n\times n$ matrix.  The $\alpha$-\emph{permanent}
and $\alpha$-\emph{determinant} of $A$ are  defined by 
$$
\per_\alpha(A)= \sum_{\sigma \in \sym_n} \alpha^{c(\sigma)}\prod_{i=1}^n a_{i\sigma(i)} \quad \mbox{ and } \quad  \det_{\alpha}(A) = \alpha^n \per_{1/\alpha}(A),
$$
where $c(\sigma)$ denotes the number of disjoint cycles of $\sigma$. Motivated by problems on random point processes,  Shirai and Takahashi \cite{Sh1,Sh2} posed the following problem:
\begin{problem}
For which $\alpha \in \RR$ is 
\begin{enumerate}
\item $\det_\alpha(A) \geq 0$ for all real symmetric positive semidefinite matrices $A$? Let $D_\RR$ be the set of such $\alpha$'s. 
\item $\det_\alpha(A) \geq 0$ for all  complex hermitian positive semidefinite matrices $A$?  Let $D_\CC$ be the set of such $\alpha$'s. 
\end{enumerate}
\end{problem}
Shirai \cite{Sh2} proved that 
$$\{\pm 1/(m+1) : m \in \NN\} \subseteq D_\CC \subseteq \{-1/(m+1) : m \in \NN\} \cup  [0,1] \quad \mbox{ and}$$ 
$$\{- 1/(m+1) : m \in \NN\}\cup \{2/(m+1): m \in \NN\}\subseteq D_\RR \subseteq \{-1/(m+1) : m \in \NN\}\cup [0,2].$$ Moreover, Shirai and Takahashi conjectured: 

\begin{conjecture}[Shirai and Takahashi, \cite{Sh1,Sh2}]\label{shirai}
\begin{align*}
D_\RR &=\{-1/(m+1) : m \in \NN\}\cup [0,2]  \mbox{ and }\\
D_\CC &= \{-1/(m+1) : m \in \NN\}\cup [0,1]. 
\end{align*}
\end{conjecture}
We shall see that Conjecture \ref{shirai} is false, infact
\begin{theorem}\label{put}
\begin{align*}
D_\RR&=\{-1/(m+1) : m \in \NN\}\cup \{2/(m+1) : m \in \NN\}\cup\{0\} \mbox{ and}\\ 
D_\CC &= \{\pm1/(m+1) : m \in \NN\}\cup\{0\} . 
\end{align*}
 \end{theorem}

The proof of Theorem \ref{put} relies on recent results on complete monotonicity due to Scott and Sokal \cite{ScSo}.

For $\nn =(n_1,\ldots, n_m) \in \NN^m$ and $A=(a_{ij})_{i,j=1}^m$, let $A[\nn]$ be the $|\nn| \times |\nn|$ matrix, where $|\nn|=\sum_{i=1}^m n_i$, obtained by replacing the 
$(i,j)$'th  entry of $A$ by an $n_i \times n_j$ matrix whose entries are all equal to $a_{ij}$. The next theorem is a generalization of the MacMahon Master Theorem. 

\begin{theorem}[Foata and Zeilberger, \cite{FZ}; Vere-Jones, \cite{VJ}]\label{macmahon}
Let $A=(a_{ij})_{i,j=1}^m$, $X=\diag(x_1,\ldots, x_m)$ and $\alpha \in \RR$. Then 
\begin{align*}
\det(I-XA)^{-\alpha} &= \sum_{\nn \in \NN^m} \per_\alpha(A[\nn])\frac {x^\nn}{\nn!} \mbox{ and }\\
\det(I-\alpha XA)^{-1/\alpha} &= \sum_{\nn \in \NN^m} \det_\alpha(A[\nn])\frac {x^\nn}{\nn!},  
\end{align*}
where $x^\nn= x_1^{n_1}\cdots x_m^{n_m}$ and $\nn!= n_1! \cdots n_m!$. 
\end{theorem}

\begin{remark}\label{nn}
If $A$ is a hermitian positive semidefinite $m \times m$ matrix and $\nn \in \NN^m$, then 
$A[\nn]$ is positive semidefinite. Indeed, if $y$ is the column vector with entries $y_{ij}$ for $1\leq i \leq m$ and $1\leq j \leq n_i$, then
$$
\overline{y}^TA[\nn]y = \overline{x}^TAx \geq 0 
$$
where $x = (x_1, \ldots, x_m)^T$ has entries $x_i= \sum_{j=1}^{n_i}y_{ij}$. 
\end{remark}

Recall that a $C^{\infty}$-function $f : \RR_{++}^m \rightarrow \RR$ is \emph{completely monotone} if 
$$
(-1)^{|\nn|}\frac {\partial^{n_1}}{\partial x_1^{n_1}} \cdots \frac {\partial^{n_m}}{\partial x_m^{n_m}}f(x) \geq 0,
$$
for all $\nn \in \NN^m$ and $x \in \RR_{++}^m$. For $m \geq 1$ let 
$$
C(m) = \NN \cup \{ x \in \RR : x \geq m-1\} \mbox{ and } R(m) = \{ x/2 : x \in C(m)\}.
$$
\begin{theorem}[Scott and Sokal, \cite{ScSo}]\label{scso}
Let $A_1,\ldots, A_n$ be $m\times m$ real or complex hermitian matrices, and form the polynomial 
$$
P(x)= \det \left( \sum_{i=1}^n x_i A_i\right).
$$
Assume that $P \not \equiv 0$ and $\beta \geq 0$. 
\begin{enumerate}
\item If  $A_1,\ldots, A_n$ are real symmetric and positive semidefinite, then $P^{-\beta}$ is completely monotone for all $\beta \in R(m)$. If $A_1, \ldots, A_n$ span the space of 
$m\times m$ symmetric matrices, then $P^{-\beta}$ fails to be completely monotone for each $\beta \not \in R(m)$. 
\item If  $A_1,\ldots, A_n$ are complex hermitian positive semidefinite, then $P^{-\beta}$ is completely monotone for all $\beta \in C(m)$. If $A_1, \ldots, A_n$ span the space of 
$m\times m$ complex hermitian matrices, then $P^{-\beta}$ fails to be completely monotone for each $\beta \not \in C(m)$. 
\end{enumerate}

\end{theorem}
We will use the following elementary but useful lemma. 
\begin{lemma}\label{flip}
Let $A$ be an $m \times n$ matrix and $B$ be an $n\times m$ matrix, then 
$$
\det(I- AB)= \det(I-BA).
$$ 
\end{lemma}

\begin{proof}[Proof of Theorem \ref{put}]
We prove Theorem \ref{put} for real symmetric matrices, the proof for complex hermitian matrices is almost identical. 

The cases left to consider is for $\alpha \in [0,2]\setminus \{2/(k+1): k \in \NN\}$.  Then $\beta =1/\alpha  \notin R(m)$ for some $m \in \NN$. Choose positive semidefinite rank one matrices $A_1, \ldots, A_n$ that span the space of real symmetric $m \times m$ matrices. Then 
$$
P(x)= \det \left( \sum_{i=1}^n x_i A_i\right)
$$
is not completely monotone by Theorem \ref{scso}. Hence there is a vector $y \in \RR_{++}^n$ such that 
$x \mapsto P(y-x)^{-\beta}$ fails to have only nonnegative Taylor coefficients. Since 
$A:= \sum_{i=1}^n y_iA_i$ is positive definite we may write $A=B^{-1/2}B^{-1/2}$ for some positive definite matrix $B$. Let $B_i = B^{1/2}A_iB^{1/2}$ and write $B_i = u_iu_i^T$ for some vector 
$u_i \in \RR^m$. Collect the $u_i$'s as columns in an $m \times n$ matrix $U$. Then, by Lemma \ref{flip}, 
\begin{align*}
P(y-x)&= \det\left(A-\sum_{i=1}^n x_i A_i\right)=\det(A)\det\left(I-\sum_{i=1}^nx_i B_i\right)\\
&=\det(A)\det(I-UXU^T)= \det(A)\det(I-XU^TU).
\end{align*}
By Theorem \ref{macmahon} and the fact that $P(y-x)^{-\beta}$ has at least one negative Taylor coefficient, there is a vector $\nn \in \NN^n$ such that $\per_\beta(U^TU[\nn]) <0$, so that $\det_\alpha(U^TU[\nn]) <0$. The matrix $U^TU[\nn]$ is positive semidefinite by Remark \ref{nn}. 
\end{proof}

\end{document}